\documentclass[oneside,english,reqno]{amsart}
\usepackage[T1]{fontenc}
\usepackage[latin9]{inputenc}
\usepackage{amsthm}
\usepackage{amssymb}
\usepackage{amsmath}

\makeatletter
\numberwithin{equation}{section}
\numberwithin{figure}{section}
\theoremstyle{plain}
\newtheorem{thm}{Theorem}[section]
\newtheorem{prop}[thm]{Proposition}
\newtheorem{definition}[thm]{Definition}

\newtheorem{lem}[thm]{Lemma}
\newtheorem{cor}[thm]{Corollary}
\newtheorem{rem}[thm]{Remark}

  \newcounter{casectr}
  
\theoremstyle{definition}

\theoremstyle{remark}

\newcommand{\RRR}{\mathbb{R}}

\newcommand{\AAA}{\mathcal{A}}

\newcommand{\EEE}{\mathcal{E}}

\newcommand{\limn}{\lim_{n\rightarrow \infty}}
\newcommand{\limsupn}{\limsup_{n\rightarrow\infty}}
\newcommand{\xtn}{x(t_{n_{0}})}
\newcommand{\xitn}{\xi({t_{n_{0}}})}
\newcommand{\ltn}{\lambda(t_{n_{0}})}
\makeatother

\usepackage{fancyhdr}
\usepackage{babel}
\providecommand{\casename}{Case}
\title{The $L^{2}$ weak sequential convergence of radial mass critical NLS solutions with mass above the ground state.  }

\begin{document}
\begin{center}
\author{Chenjie Fan $^{\ddagger} $ $^\dagger$}
\end{center}
\thanks{\quad \\$^{\ddagger} $Department of Mathematics,  Massachusetts Institute of Technology,  77 Massachusetts Ave,  Cambridge,  MA 02139-4307 USA. email:
cjfan@math.mit.edu.\\
$^\dagger$ The author is partially supported by NSF Grant DMS 1069225, DMS 1362509 and DMS 1462401.}
\maketitle
\begin{abstract}
We study the non-scattering $L^{2}$ solution  $u$ to the  radial mass critical nonlinear Schr\"odinger equation with mass just above the ground state, and show that there exists a time sequence $\{t_{n}\}_{n}$, such that  $u(t_{n})$ weakly converges to the ground state $Q$ up to  scaling and  phase transformation. We also give some partial results on the mass concentration of the minimal mass blow up solution.
\end{abstract}
\section{Introduction}
In this work, we study the Cauchy Problem to the focusing mass critical NLS at  regularity  $L^{2}$
\begin{equation}\label{nls}
\begin{cases}
iu_{t}+\Delta u= - |u|^{4/d}u, \\
u(0,x)=u_{0}\in L^{2}(\RRR^{d}).\\
\end{cases}
\end{equation}

Here $d$ denotes the dimension. We assume $d=1,2,3$ to reduce technicalities, but the results should hold for general dimension as well.


Equation \eqref{nls} has three  conservation laws:
\begin{itemize}
\item Mass:
\begin{equation}
M(u(t,x)):=\int |u(t,x)|^{2}dx=M(u_{0}),
\end{equation}
\item Energy:
\begin{equation}
E(u(t,x)):=\frac{1}{2}\int |\triangledown u(t,x)|^{2}dx-\frac{1}{2+\frac{4}{d}}\int|u(t,x)|^{2+\frac{4}{d}}dx=E(u_{0}),
\end{equation}
\item Momentum:
\begin{equation}
P(u(t,x)):=\Im(\int \triangledown u(t,x) \overline{u(t,x)})dx=P(u_{0}),
\end{equation}
\end{itemize}
and  the following symmetries:
\begin{enumerate}
\item Space-time translation: If $u(t,x)$ solves \eqref{nls}, then $\forall t_{0}\in \RRR, x_{0}\in \RRR^{d}$, we have  $u(t-t_{0}, x-x_{0})$ solves \eqref{nls}.
\item Phase transformation: If $u$ solves \eqref{nls}, then $\forall \theta_{0}\in \RRR$, we have $e^{i\theta_{0}}u$ solves \eqref{nls}. 
\item Galilean transformation: If $u(t,x)$ solves \eqref{nls}, then $\forall \xi\in \RRR^{d}$, we have  $u(t,x-\xi t)e^{i\frac{\xi}{2}(x-\frac{\xi}{2}t)}$ solves \eqref{nls}.
\item Scaling: If $u(t,x)$ solves \eqref{nls},  then $\forall \lambda \in \RRR_{+}$,  we have $u_{\lambda}(t,x):=\frac{1}{\lambda^{\frac{d}{2}}}u(\frac{t}{\lambda}, \frac{x}{\lambda})$ solves \eqref{nls}.
\item Pseudo-conformal transformation: If $u(t,x)$ solves \eqref{nls}, then $\frac{1}{t^{\frac{d}{2}}}\bar{u}(\frac{1}{t},\frac{x}{t})e^{i\frac{|x|^{2}}{4t}}$ solves \eqref{nls}.
\end{enumerate}

Throughout the whole article, we assume 

\begin{equation}\label{supercritical}
\|Q\|_{2}\leq\|u_{0}\|_{2}\leq \|Q\|_{2}+\alpha.
\end{equation}

\bigskip
Here $\alpha$ is some universal small constant, which will be determined later and 
$Q$ is the unique $L^{2}$, positive solution to the elliptic PDE
\begin{equation}\label{groundstate}
-\Delta Q+Q=|Q|^{4/d}Q.
\end{equation}
We remark $Qe^{it}$ solves \eqref{nls} if and only if $Q$ solves \eqref{groundstate}. $Q$ is smooth and decays exponentially.

\subsection{Main Results}
In this work we show the following:
\begin{thm}\label{main1}
Assume $u$ is a radial  solution to \eqref{nls}, with $u_0$ satisfying \eqref{supercritical}, which does not scatter forward.  Let $(T^{-}(u),
 T^{+}(u))$ be its lifespan,  then there exist a time sequence $t_{n}\rightarrow T^{+}(u)$, and a family of parameters $\lambda_{*,n},  , \gamma_{*,n}$ such that
\begin{equation}
\lambda_{*,n}^{d/2}u(t_{n}, \lambda_{*,n}x)e^{-i\gamma_{*,n}}\rightharpoonup Q \text{ in } L^{2}.
\end{equation}
\end{thm}
See Definition \ref{defnscatter} for the precise notion of scattering forward.

\bigskip
If one further assume $\|u\|_{2}=\|Q\|_{2}$, then we can upgrade the above to
\begin{thm}\label{main2}
Assume $u$ is a radial solution to \eqref{nls}, with  $\|u\|_{2}=\|Q\|_{2}$, which does not scatter forward.  Let $(T^{-}(u),
 T^{+}(u))$ be its lifespan, 
 then there exists a sequence $t_{n}\rightarrow T^{+}$, and a family of parameters $\lambda_{*,n}, \gamma_{*,n}$ such that
 \begin{equation}
\lambda_{*,n}^{d/2}u(t_n,\lambda_{*,n}x)e^{-i\gamma_{*,n}}\rightarrow Q \text{ in } L^{2}.
\end{equation}
\end{thm}

\begin{rem}\label{fr}
 Most of the proof for Theorem \ref{main1} and  Theorem \ref{main2}, written in this work can be obtained for  the nonradial case as well. Indeed, only one step (Lemma \ref{seeyou}) cannot be generalized to the nonradial case. In particular, we do not use Sobolev embedding or weighted Strichartz estimate for radial solutions. Moreover,  our  results hold in fact for solutions which are symmetric across any $d$ linearly independent hyperplanes. Nevertheless, the idea in this work is not enough to  cover the nonradial case. We will investigate this case  in a future work.
\end{rem}

We also obtain some partial results for the minimal mass blow up solution to \eqref{nls} at  regularity $L^{2}$, not necessarily radial.
\begin{thm}\label{aux}
Let $u$ be a general $L^{2}$ solution to \eqref{nls} that  blows up at finite time $T$ and such that  $\|u(t)\|_{2}=\|Q\|_{2}$. Then there exist sequencies $x_{n}, t_{n}$ such that
\begin{equation}\label{conrate}
\int_{|x-x_{n}|\leq (T-t_{n})^{2/3-}}|u|^{2}\geq \|Q\|_{2}.
\end{equation}
\end{thm}
We will give the proof of Theorem \ref{aux} in Appendix \ref{sectionaux}.
\begin{rem}
If one further assumes that  the initial data is in $H^{1}$ and with same mass as $Q$, then \eqref{conrate}  holds even if  one changes the power $2/3$ to $1$. Indeed, the $H^{1}$ minimal mass blow up solution  is  determined by \cite{merle1993determination} and can be written down explicitely. 
\end{rem}
\begin{rem}
Though not explicitly stated in the literature, it is not hard to combine concentration compactness and Dodson's scattering result in \cite{dodson2015global}, (see also Theorem  \ref{dodson}), to show that if $u$ is a general solution to the \eqref{nls},   which blows up in finite time $T$,  then there exist $t_{n}\rightarrow T$, and $x_{n}\in \RRR^{d}$, such that
\begin{equation}
\int_{|x-x_{n}|\leq (T-t_{n})^{1/2-}} |u|^{2}\geq \|Q\|_{2}^{2}.
\end{equation}
However, as far as we are concerned, it is always of interest to improve the above  estimate to  at least
\begin{equation}\label{conj}
\int_{|x-x_{n}|\leq (T-t_{n})^{1/2}} |u|^{2}\geq \|Q\|_{2}^{2}.
\end{equation}
Estimate \eqref{conj}, for $u_{0}\in H^{1}$ and satisfying \eqref{supercritical},  should follow using a rigidity result due to Rapha\"el \cite{raphael2005stability}, which is highly nontrivial. 
\end{rem}
\subsection{Background}

Equation \eqref{nls} is called focusing since its associated energy is not coersive, and it is called mass critical since its  mass conservation law is invariant under scaling symmetry. 
Both focusing and defocusing NLS are locally well posed in $L^{2}$, \cite{cazenave1989some}, see also \cite{cazenave2003semilinear}, \cite{tao2006nonlinear}. And it is an active research topic to understand the long time dynamic.

It is known that  the solution to the  defocusing NLS with arbitrary $L^{2}$ data does not break down and scatters to some linear solution, \cite{dodson2012global}.

On the other hand, the focusing problem, \eqref{nls}, is known to have  more complicate dynamics,  and the solution may  break down in finite time, \cite{glassey1977blowing}.  It is of great interest to understand the blow up phenomena, rather than just showing the existence of blow up.  We remark that since we are working on the critical space $L^{2}$ and pseudo conformal symmetry\footnote{This symmetry, however, is not a symmetry in $H^{1}$.} is  a symmetry in $L^{2}$,  finite time blow up solutions and non-scattering solutions are essentially the same.

We recall that the ground state $Q$ gives a threshold of scattering dynamic. Indeed, it is known that any solution to \eqref{nls} scatters to a linear solution if  the initial data has mass strictly below the ground state, \cite{dodson2015global}. See also previous work \cite{weinstein1983nonlinear}.

The main purpose of this article is to  understand the possible long time dynamic of \eqref{nls} for solution with mass at or just above the threshold $\|Q\|_{2}^{2}$, hence assumption \eqref{supercritical}.

Under assumption \eqref{supercritical}, the finite time blow up\footnote{Note that finite time blow up should be understood as one of the  long time dynamics rather than a short time dynamic.} solution to \eqref{nls} at  regularity $H^{1}$ has been extensively studied in recent years. We recall the work of  \cite{landman1988rate}\cite{perelman2001blow}, \cite{merle2005blow}, \cite{merle2003sharp} , \cite{merle2006sharp}, \cite{merle2004universality}, \cite{raphael2005stability}, \cite{merle2005profiles} regarding the so-called log-log blow up dynamic.  If one assumes the initial data is in $H^{1}$, with negative energy and statisfies assumption \eqref{supercritical}, then one can upgrade the sequential convergence in Theorem \ref{main1} to convergence as $t\rightarrow T^{+}$, \cite{merle2004universality}.  It is also shown in \cite{merle2004universality} that  for general  $H^{1}$ solution to \eqref{nls} satisfying \eqref{supercritical}, without the sign condition in the  energy, Theorem  \ref{main1} also holds.   Regarding Theorem \ref{main2}, if one assumes the initial data is in $H^{1}$ and blows up in finite time, then the convergence holds as $t\rightarrow T^{+}$ thanks to Merle's complete classification of minimal mass blow up solutions. Moreover, in Theorem \ref{main2}, if one assumes $u$ is radial and in $H^{1}$, $d=2,3$, and $u$ is global, $\|u\|_{2}=\|Q\|_{2}$, then one can still obtain  convergence as $t\rightarrow T^{+}$, due to \cite{li2012rigidity}. \cite{li2012rigidity} indeed shows such solution must be a solitary wave. When  assumes $d\geq 4$, (our work mainly deals with $d=1,2,3$), if one assumes $u$ is radial and  in $L^{2}$, $\|u\|_{2}=\|Q\|_{2}$, global in both sides, then \cite{li2010regularity} shows  such solution must be a solitary wave.

The main purpose of this work is to extend these results, (except the 4 dimensional result \cite{li2010regularity}), to the lower $L^{2}$ regularity.

 We point out if one only wants to show  some sequential weak convergence but does  not want to characterize the limit profile, then assumption \eqref{supercritical} or radial assumption may not be necessary, and  the method in \cite{dkmnew} should be  applicable to prove these kinds of results.  Indeed, it is pointed out in \cite{dkmnew}, their methods and strategy should be able to handle general dispersive equations once a suitable profile decomposition is available. However, assumption \eqref{supercritical} and the radial assumption are very important for us to determine that the limit profile is indeed $Q$. 

Results  with type similar to Theorem \ref{main1}, Theorem \ref{main2} will also appear naturally when one consider the mass concentration phenomena of finite time blow up solutions to \eqref{nls}, we refer to \cite{merle19902}, \cite{cazenave2003semilinear}, \cite{nawa1990mass},\cite{hmidi2005blowup}, \cite{colliander2004ground}, \cite{tzirakis2006mass}, \cite{visan2007blowup}, \cite{bourgain1998refinements}, \cite{keraani2006blow}, \cite{begout2007mass} and reference in their works.

Finally, we point out that our work is also motivated by the recent progress in the solition resolution conjecture for the energy critical wave, (where more complete and general results are available), see \cite{duyckaerts2009universality}, \cite{dkm2}, \cite{dkm3}, \cite{dkmcompact},\cite{dkmprofile}, \cite{dkmnew},\cite{cote2014profiles},\cite{duyckaerts2016concentration},
 \cite{duyckaerts2016soliton}, \cite{duyckaerts2016scattering} and the references therein.

\subsection{Notation}
Throughout this work, $\alpha$ is used to denote a universal small number, $\delta(\alpha)$ is  a small number depending on $\alpha$ such that $\lim_{\alpha\rightarrow 0}\delta(\alpha)=0$. We use $C$ to denote a large constant, it usually changes line by line.

We write  $A\lesssim B$ when 
 $A\leq CB$, for some universal constant $C$,  we write $A\gtrsim B$ if $B\lesssim A$. We write $A\sim B$ if $A\lesssim B$ and $B\lesssim A$. As usual, $A\lesssim_{\sigma} B$ means that $A\leq C_{\sigma} B$, where $C_\sigma$ is a constant depending on $\sigma$
 
We  use the usual functional spaces $L^{p}$,  we will also use the Sobolev space $H^{1}$. We sometimes  write $L^{p}$ for $L^{p}(\RRR^{d})$, and similarly  for the  other spaces.  We also use $L_{t}^{q}L_{x}^{p}$ to denote $L^{q}(\RRR; L^{p}(\RRR^{d}))$.   When a certain function is only defined on $I\times \RRR^{d}$, we  also use the notation $L^{q}(I;L^{p}(\RRR^{d}))$.  Sometimes we use $\|f\|_{p}$ to denote $\|f\|_{L^{p}}$. We also use $\|u\|_{2}$ to denotes the $\|u(0)\|_{2}$, when $u$ is a solution to the Schr\"odinger equation with initial data $u_0$, since  $L^{2}$ is preserved under Schr\"odinger flow.

We use the standard Littlewood-Paley projection operator $P_{<N}, P_{>N}$, we quickly  recall the definition here. Let $\psi$ be a bump function which equals to 1 when $|x|\leq 1$, and vanishes for $|x|\geq 2$, one define the multiplier $P<_{N}$ as 
\begin{equation}
\widehat{P_{N}f}(\xi):=\psi(\frac{\xi}{N})\hat{f}(\xi).
\end{equation} 
And$ P_{>N}=1-P_{<N}.$

We use $<,>$ to denote the usual $ L^{2}$ (complex) inner product.

Finally, for a solution $u(t,x)$, we use $(T^{-}(u),T^{+}(u))$ to denote its lifespan.

\section{Preliminary}
We present the preliminary for this work, experts may skip this section in the reading.
\subsection{Local well posedness (LWP) and stability}
\subsubsection{Classical Strichartz estimates}
The local well poseness of \eqref{nls} is established using the classical Strichartz Estimates. We recall them below.
Consider the linear Schr\"odinger equation:
\begin{equation}\label{ls}
\begin{cases}
iu_{t}+\Delta u=0,\\
u(0,x)=u_{0}\in L^{2}(\RRR^{d}).
\end{cases}
\end{equation}
We use $e^{it\Delta}$ to denote the linear propagator.  One has estimates
\begin{eqnarray}\label{stri}
&&\|e^{it\Delta}u_{0}\|_{L^{2(d+2)/d}_{t,x}}\lesssim \|u_{0}\|_{L^{2}_{x}},\\
&&\left\lVert \int_{0}^{t}e^{i(t-s)\Delta}f(s,x)ds\right\lVert_{L_{t,x}^{2(d+2)/d}\cap L^{\infty}L^{2}}\lesssim \|f\|_{L_{t,x}^{\frac{2(d+2)}{d+4}}}.
\end{eqnarray}
We refer to \cite{cazenave2003semilinear} \cite{keel1998endpoint},\cite{tao2006nonlinear}  and reference therein for a proof.
\begin{rem}
Strichartz estimate  holds in more general case, indeed one has 
\begin{equation}
\left\lVert \int_{0}^{t}e^{i(t-s)\Delta}f(s,x)ds\right\lVert_{L_{t}^{q}L_{x}^{r}}\lesssim \|f\|_{L_{t}^{\tilde{q}'}L_{x}^ {\tilde{r}'}},
\end{equation}
where $(q,r)$, $(\tilde{q},\tilde{r})$ is admissible in the sense $\frac{2}{q}+\frac{d}{r}=\frac{d}{2}$ and $(q,r,d), (\tilde{q}, \tilde{r},d)\neq (2,\infty,2)$. We fix $(q,r)=(\tilde{q}, \tilde{r})=(\frac{2(d+2)}{d}, \frac{2(d+2)}{d})$ for simplicity. Similarly, the local well posedness and stability holds in more general sense.
\end{rem}
In the rest of this section, we quickly recall the classical results in the local well posedness theory without proof. We refer to \cite{cazenave1988cauchy}, \cite{cazenave1990cauchy}, \cite{tao2006nonlinear}, \cite{cazenave2003semilinear} and the reference there in  for a proof.
\subsubsection{Local existence}
\begin{thm}
Given $u_{0}$ in $L^{2}$, there exists  $T=T(u_{0})$ such that there is a unique solution $u(t,x)$ to \eqref{nls}  with initial data $u_{0}$ in the following sense:
\begin{equation}\label{duhamel2}
u(t,x)=e^{it\Delta}u_{0}+i\int_{0}^{t}e^{i(t-\tau)\Delta}(|u|^{4/d}u(\tau))d\tau, \quad t\in [0,T]
\end{equation}
Formula \eqref{duhamel2} holds in space $C([0,T];L^{2})\cap L_{t,x}^{2(d+2)/d}([0,T]\times \RRR^{d})$.
\end{thm}
\subsubsection{Blow up criteria}
We have the following blow up criteria.
\begin{prop}
Given a solution $u$ with initial data $u_{0}$, assume $(T_{-},T_{+})$ is the maximal time interval $u$ can be defined on, if $T_{+}$ is finite, then
\begin{equation}
\|u\|_{L^{2(d+2)/d}[0,T_{+})\times \RRR^{d}}=\infty
\end{equation} 
Similarly results holds for $T_{-}$.
\end{prop}
Now, we can define the notion of scattering.
\begin{definition}\label{defnscatter}
We say a solution $u$ to \eqref{nls} scatters forward  if $T^{+}=\infty$ and $\|u\|_{L^{2(d+2)/d}[0,T_{+})\times \RRR^{d}}<\infty.$.
Similarly, we define the notion of scattering backward. If $u$ scatters both backward and forward, then we say $u$ scatters.
\end{definition}
\subsubsection{Small data theory}
\begin{prop}\label{smalldata}
There exists $\epsilon_{0}>0$ such that
if the initial data $u_{0}\in L^{2}$ and  satisfy
\begin{equation}\label{small}
\|e^{it\Delta}u_{0}\|_{L^{2(d+2)/d}_{t,x}}\leq \epsilon_{0}
\end{equation}
 then the solution to \eqref{nls} with initial data $u_{0}$ is global and one has estimate
\begin{equation}
\|u(t,x)\|_{L^{2(d+2)/d}_{t,x}}\lesssim \|u_{0}\|_{2}.
\end{equation}
\end{prop}
\begin{rem}\label{remsmalldata}
By Strichartz estimate \eqref{stri}, it is clear that \eqref{small} holds when $\|u_{0}\|_{2}$ is small enough.
\end{rem}
\subsubsection{Stability}
We state a stability result  about \eqref{nls}. This kind of argument is  standard nowadays. One may  refer to Lemma 3.9, lemma 3.10 in \cite{colliander2008global}.
\begin{prop}\label{stability}
Let $I$ be a compact interval, and  $0\in I$. Let $\tilde{u}$ is a near-solution to \eqref{nls} in the sense
\begin{equation}
i\tilde{u}_{t}+\Delta \tilde{u}+|\tilde{u}|^{4/d}\tilde{u}=e,
\end{equation}
and the following estimate holds
\begin{eqnarray}
\|\tilde{u}\|_{L^{2(d+2)/d}(I\times \RRR)}\leq M,\\
\|\tilde{u}\|_{L_{I}^{\infty}L^{2}(\RRR)}\leq E,\\
\|e\|_{L^{\frac{2(d+2)}{d+4}}(I\times \RRR)}\leq \epsilon,
\end{eqnarray}
where $\epsilon<\epsilon_{1}:=\epsilon_{1}(M_{1},M_{2})$.
Assume further there exists $t_{0}\in I$ and $u_{0} \in L^{2}$ such that
\begin{equation}
\|\tilde{u}(t_{0})-u_{0}\|_{2}<\epsilon,
\end{equation}
Then, there exists a unique solution $u(t,x)$ to the Cauchy Problem
\begin{equation}
\begin{cases}
iu_{t}+\Delta u+|u|^{4/d}u=0,\\
u(t_{0})=u_{0}.
\end{cases}
\end{equation}
such that
\begin{eqnarray}
\|u-\tilde{u}\|_{L^{\frac{2(d+2)}{d}}(I\times \RRR)\cap L_{I}^{\infty}L^{2}(\RRR)}\l \lesssim_{M_{1},M_{2}} \epsilon.
\end{eqnarray}

In particular
\begin{equation}
\|u\|_{L^{2(d+2)/d}(I\times \RRR)}\lesssim_{M_{1},M_{2}} 1.
\end{equation}
\end{prop}
\subsection{Scattering below the mass of the ground state}
The dynamic of \eqref{nls} for initial data with mass blow up the groud state is known, due to the following Theorem by Dodson ,\cite{dodson2015global}.
\begin{thm}[Dodson]\label{dodson}
Consider the Cauchy Problem \eqref{nls}, with the initial data $u_{0}$ such that
\begin{equation}\label{restriction}
\|u_{0}\|_{L_{x}^{2}}<\|Q\|_{L_{x}^{2}}.
\end{equation}
 The solution $u$ to \eqref{nls} is global, further more it scatters 
\begin{equation}\label{scattering}
\|u(t,x)\|_{L^{2(d+2)/d}_{t,x}}<\infty.
\end{equation}
\end{thm}

\subsection{Concentration compactness}
Strichartz estimates \eqref{stri} lacks compactness due to the symmetry of equation \eqref{ls}. Note the aforementioned symmetries for the nonlinear equation \eqref{nls} also hold for the linear equation. Profile decomposition is the tool to remedy this.

Let us start with the following definition,
\begin{definition}\label{group}
Let $G:=\{g=g_{x_{0},t_{0},\xi_{0},\lambda_{0}}| x_{0},t_{0},\xi_{0}\in \RRR, \lambda_{0}\in \RRR_{+}\}$, where
$g_{x_{0},\xi_{0},\lambda_{0},\xi_{0}}$ is a map $L^{2}(\RRR)\rightarrow L^{2}(\RRR)$:
\begin{equation}\label{groupdefi}
g_{x_{0},\xi_{0},\lambda_{0},t_{0}}f(x):=\frac{1}{\lambda_{0}^{\frac{d}{2}}}e^{ix\xi_{0}}(e^{i(-\frac{t_{0}}{\lambda_{0}^{2}})\Delta}f)(\frac{x-x_{0}}{\lambda_{0}}).
\end{equation}
\end{definition}
\begin{rem}\label{unitary}
$G$ is a group acting on $L^{2}$ and  for any $g\in G, f\in L^{2}$, $\|g\cdot f\|_{2}=\|f\|_{2}$.
\end{rem}
\subsubsection{Profile decomposition}
One has
\begin{prop}[Theorem 5.4 \cite{begout2007mass},  Theorem 4.2 \cite{tao2008minimal}, Theorem 2 \cite{merle1998compactness}]\label{profiledecomposition}
Let $\{u_{n}\}_{n=1}^{\infty}$ be bounded in $L^{2}(\RRR^{d})$, then up to extracting subsequence, there exist a family of $L^{2}$ functions $\phi_{j}$, 
$j=1,2,\cdots$ and group elements $g_{j,n}\in G$ , where  $g_{j,n}=(g_{j,n})_{x_{j,n},\xi_{j,n},\lambda_{j,n},t_{j,n}}$, such that for all $ l=1,2,\cdots$ we have the decomposition
\begin{equation}\label{tempdeco}
u_{n}= \sum_{j=1}^{l}g_{j,n}\phi_{j}+\omega_{n}^{l},
\end{equation}
(here \eqref{tempdeco} defines $\omega_{n}^{l}$.)
And the following properties hold
\begin{itemize}
\item Asymptotically orthogonality of the group elements:
For any $j\neq j'$,  one has 
\begin{equation}\label{orthpara}
\frac{\lambda_{j,n}}{\lambda_{j',n}}+\frac{\lambda_{j',n}}{\lambda_{j,n}}+|\frac{x_{j,n}-x_{j',n}}{\lambda_{j,n}}|+|\lambda_{j,n}(\xi_{j,n}-\xi_{j',n})|+|\frac{t_{j,n}-t_{j',n}}{\lambda_{j,n}^{2}}|\xrightarrow{n\rightarrow \infty}\infty
\end{equation}
\item Asymptotically orthogonality of mass:
\begin{equation}\label{orthmass}
\forall l\geq 1, \lim_{n\rightarrow \infty}|M(u_{n})-\sum_{j\leq l}M(\phi_{j})-M(\omega^{l}_{n})|=0.
\end{equation}
\item Asymptotically orthogonality of Strichartz norm:
\begin{equation}\label{orthstri}
\forall j\neq j', \lim_{n\rightarrow \infty}\|e^{it\Delta}(g_{j,n}\phi_{j})e^{it\Delta}(g_{j',n}\phi_{j'})\|_{L_{t,x}^{(d+2)/d}}=0.
\end{equation}
\item Smallness of remainder term:
\begin{equation}\label{smallnessofmass}
\lim_{l\rightarrow\infty}\limsupn \|e^{it\Delta}\omega_{n}^{l}\|_{L_{t,x}^{2(d+2)/d}}=0
\end{equation}
\item Weak limit condition of the remainder term:
\begin{equation}\label{weakconve}
\forall j\leq l, g_{j,n}^{-1}\omega^{l}_{n}\rightharpoonup 0 \text{ in } L_{x}^{2}
\end{equation}
\end{itemize}

\end{prop}

According to Proposition \ref{profiledecomposition},  we define 
\begin{definition}
A linear profile is a function $f\in L^{2}$ and a sequence $\{g_{n}\}_{n}\subset G$, or equivalently a function $f\in L^{2}$ with parameters  $\{x_{n},\xi_{n},\lambda_{n},t_{n}\}_{n}\subset \RRR\times \RRR \times \RRR_{+}\times \RRR $.
\end{definition}

Up to extracting subsequence and adjusting the profile,  for every profile decomposition as in Proposition \ref{profiledecomposition},  we always assume without loss of generality that $\limn-\frac{-t_{j,n}}{\lambda_{j,n}^{2}}=0$ or equal to $\pm \infty$. This leads to the following standard definition:
\begin{definition}
We call a profile $f$ with parameter $\{x_{n},\xi_{n},\lambda_{n},t_{n}\}_{n}$
\begin{itemize}
\item Compact profile if $t_{n}\equiv 0$,
\item Forward scattering profile if $-\frac{t_{n}}{\lambda_{n}^{2}}=\infty$,
\item Backward scattering profile if $-\frac{t_{n}}{\lambda_{n}^{2}}=-\infty$.
\end{itemize}
\end{definition}

\subsubsection{Nonlinear approximation}
To deal with the nonlinear equation \eqref{nls}, one needs the notion of nonlinear profile.

\begin{definition}
Given a linear  profile, i.e. a function $f\in L^{2}$ and a sequence $\{g_{n}\}_{n}\subset G$, or equivalently a function $f\in L^{2}$ with parameters  $\{x_{n},\xi_{n},\lambda_{n},t_{n}\}_{n}\subset \RRR^{d}\times \RRR^{d} \times \RRR_{+}\times \RRR $.
We say $U$ is  the nonlinear profile associated with this linear profile if $U$ is a solution to nonlinear Schr\"odinger equation
\begin{equation}
iU_{t}+\Delta U+|U|^{4/d}U=0,
\end{equation}
and satisfy the estimate
\begin{equation}\label{propnonlinear}
\limn \left\|U(-\frac{t_{n}}{\lambda_{j,n}^{2}})-e^{i(-\frac{t_{n}}{\lambda_{n}^{2}})\Delta}f \right\|_{2}=0.
\end{equation}
 \eqref{propnonlinear} is only required for n large enough. For \eqref{propnonlinear} to make sense, we require U is defined in a neighborhood \footnote{In profile decomposition, one usually needs to  extract subsequence many times,  we always assume that $\limn -\frac{t_{n}}{\lambda_{n}^{2}}$ exits or equal to $\pm \infty$. We also define the neighborhood of $\infty$ as $(M,\infty)$ for any M, similarly we define the neighborhood of $-\infty$.}of $\limn -\frac{t_{n}}{\lambda_{n}^{2}}$. 
\end{definition}
\begin{rem}
Given a linear profile, the associated  nonlinear profile exists and is unique. The existence and uniqueness of the nonlinear profile basically relies on the local well posedness theory of \eqref{nls}. This is quite standard, see Notation 2.6 in \cite{duyckaerts2009universality}.
\end{rem}

Now, we state a nonlinear approximation for NLS. Though we do not give the detailed proof here, we point out it  is essentially the consequence of asymptotically  orthogonality \eqref{orthpara} and the classical stability theory Proposition \ref{stability}, see  \cite{begout2007mass},\cite{merle1998compactness}, \cite{tao2008minimal} for more details.
\begin{prop}\label{nonlinearprofile}
Assume that $\{u_{n}\}_{n}$ admits  a profile decomposition with profiles
 $\left(\phi_{j};\{x_{j,n},\xi_{j,n},\lambda_{j,n},t_{j,n}\}_{n}\right)_{j}$ as in Proposition \ref{profiledecomposition}. Let us consider a sequence of  Cauchy problems for a sequence of initial time $\{t_{n}\}$,
\begin{equation}\label{temp}
\begin{cases}
i\partial_{t}v_{n}(t,x)+\Delta v_{n}(t,x)+|v_{n}|^{4}v_{n}(t,x)=0,\\
v_{n}(t_{n},x)=u_{n}.
\end{cases}
\end{equation}
For each fixed j, let $\Phi_{j}(t,x)$ be the associated  nonlinear profile to $\phi_{j}$.  Let $$\Phi_{j,n}:=\frac{1}{\lambda_{j,n}^{\frac{d}{2}}}e^{ix\xi_{j,n}}e^{-it|\xi_{j,n}|^{2}}\Phi_{j}\left(\frac{t-t_{j,n}}{\lambda_{j,n}^{2}},\frac{x-x_{j,n}-2\xi_{j,n}t}{\lambda_{j,n}}\right).$$

If for any $\{\tau_{n}\}_{n}$ such that:
\begin{equation}\label{criticalnonlinear}
\forall j>0,  \quad \Phi_{j} \text{ scatters forward or } \limn \frac{\tau_{n}-t_{j,n}}{\lambda_{j,n}^{2}}<T_{+}(\Phi_{j}),
\end{equation}
(One may extract a subsequence again so that the above limit exists),

then for n large enough,  $v_{n}$, $\Phi_{j,n}$ are defined in $[t_{n},t_{n}+\tau_{n}]$.

Moreover, let
\begin{equation}\label{keynonlineardispersive}
r_{n}^{l}:=v_{n}-\sum_{j=1}^{l}\Phi_{j,n}-e^{it_{n}\Delta}\omega_{n}^{l},
\end{equation}

then one has 
\begin{equation}
\limsup_{ l \rightarrow \infty}\limsupn \|r_{n}^{l}\|_{L^{\infty}([t_{n},t_{n}+\tau_{n}];L^{2})\cap L^{2(d+2)/d}([t_{n}, t_{n}+\tau_{n}]\times \RRR)}=0.
\end{equation}

Furthermore, \eqref{orthpara} implies that $\forall l>0$
\begin{equation}
\begin{aligned}
&\|\sum_{j=1}^{l}\Phi_{j,n}\|_{L^{2(d+2)/d}([t_{n},t_{n}+\tau_{n}]\times \RRR^{d})}^{\frac{2(d+2)}{d}}\\
\lesssim &\sum_{j=1}^{l}\|\Phi_{j}\|_{L^{2(d+2)/d}([\lim\frac{-t_{j,n}}{\lambda^{2}_{j,n}},\lim \frac{\tau_{n}-t_{j,n}}{\lambda_{j,n}}]\times \RRR^{d})}^{2(d+2)/d}+o_{n}(1)
\end{aligned}
\end{equation}
In particular, if for any $j$, $\Phi_{j}$ scatters forward, then for $n$ large enough, the associated solution to \eqref{nls} with initial data $u(t_{n})$ also scatters forward.
\end{prop}

We remark here similar results holds for energy critcal wave and energy critical NLS, see \cite{bahouri1999high}, \cite{keraani2001defect}.

\subsection{Variational characiterization of groud state}
Let us first recall the classical Gagliardo-Nirenberg inequality.
\begin{lem}\label{lemsharpGN}
Let $v\in H^{1}$, then
\begin{equation}\label{sharpGN}
E(v)\geq \frac{1}{2}\int |\nabla v|^{2}\left[1-\left(1-\frac{\|v\|_{2}^{2}}{\|Q\|_{2}^{2}}\right)^{4/d}\right]
\end{equation}
\end{lem}
It gives the variational characterization of the ground state $Q$.
\begin{lem}\label{lemvarchar}
Let $v \in H^{1}$, and 
\begin{equation}
\int |v|^{2}=\int Q^{2}, \quad  E(v)=0
\end{equation}
then
\begin{equation}
v(x)=\lambda_{0}^{N/2}Q(\lambda_{0}x+x_{0})e^{i\gamma_{0}}.
\end{equation}
\end{lem}
In \cite{merle2005blow}, Merle and Rapha\"el apply  concentration compactness type techniques to generalize the above into the following:
\begin{lem}\label{mrvar}
Let $v$ in $H^{1}$, there exists $\alpha_{0}>0$, such that for all $\alpha<\alpha_{0}$, there exists $\delta(\alpha)>0$, such that if  
\begin{equation}
\|v\|_{2}\leq \|Q\|_{2}+\alpha, \quad \quad E(v)\leq \alpha \|\nabla v\|_{2}^{2},
\end{equation}
then there exits $\lambda_{0}=\frac{\|\nabla Q\|_{2}}{\|\nabla u\|_{2}}, x_{0}\in\RRR^{d}, \gamma_{0}\in \RRR$ such that
\begin{equation}
\|\frac{1}{\lambda^{d/2}_{0}}u(\frac{x-x_{0}}{\lambda_{0}})e^{i\gamma_{0}}-Q\|_{H^{1}}\leq \delta(\alpha).
\end{equation}
and, $\lim_{\alpha\rightarrow 0}\delta(\alpha)=0.$
\end{lem}
\section{The dynamic of non-positive energy solution}
Throughout this section, we assume the solution $u$ to \eqref{nls} has $H^{1}$ initial data and satisfies \eqref{supercritical}.

The dynamics of non-positive energy solutions are  extensively studied in the series work of Merle and Rapha\"el, \cite{merle2005blow}, \cite{merle2003sharp} , \cite{merle2006sharp}, \cite{merle2004universality}, \cite{raphael2005stability}, \cite{merle2005profiles}. We will apply their results  in this work. However,  we will work  from an $L^{2}$ based viewpoint rather than $H^{1}$ viewpoint.

Merle and Rapha\"el show all strictly negative energy solutions blows up according to the so-called log-log dynamic, we restate their results as the following,

\begin{thm}\label{mr1}
Assume $u$ is a solution to \eqref{nls} with $H^{1}$ initial data, non positive energy, and satisfying assumption \eqref{supercritical}, assume further $\|u\|_{2} \neq \|Q\|_{2}$ if $u$ is of zero energy,    then $u$ blows up in finite time according to the so-called log-log law
\begin{equation}\label{loglog}
u(t,x)=\frac{1}{\lambda^{d/2}(t)}(Q+\epsilon)(\frac{x-x(t)}{\lambda})e^{i\gamma(t)}, x(t)\in \RRR^{d}, \gamma(t)\in \RRR, \lambda(t)\in \RRR^{+},
\|\epsilon\|_{H^{1}}\leq \delta(\alpha)
\end{equation}
with estimate
\begin{equation}\label{rate}
\lambda(t)\sim \sqrt{\frac{T-t}{\ln|\ln T-t}|}.
\end{equation}
\begin{equation}\label{remainder}
\lim_{t\rightarrow T} \int ( |\nabla \epsilon(t,x)|^{2}+\epsilon(t,x)|^{2}e^{-|x|})=0.
\end{equation}
\end{thm}
\begin{rem}
We do not directly  use \eqref{rate}.   Our results rely  on the fact such solution will blow up in finite time and   the mechanism of  blow up is ejecting mass out of the singular point, i.e. the control \eqref{remainder}.
\end{rem}

One may refer to \cite{merle2005blow},\cite{merle2003sharp}, \cite{merle2006sharp} for a full proof of Theorem \ref{mr1} when the solution is of strictly negative energy.  When the solution if of zero energy, one may refer to Theorem 3 in \cite{merle2006sharp}, see also  Theorem 4 and  Proposition 5 in \cite{merle2004universality}. For estimate \eqref{remainder}, which is most relevant to our work, one may refer to the formula above (3.7) in page 52 of \cite{merle2006sharp}. Strictly speaking, the term appears in \cite{merle2006sharp} is $\tilde{Q}_{b}$ rather than the ground state $Q$, but $\tilde{Q}_{b}$ is  just small modification of $Q$, and converges to $Q$ in a strong way as $b\rightarrow 0$, see Proposition 1 in \cite{merle2006sharp}. And $b\rightarrow 0$ as $t\rightarrow T$, see, again, the formula above (3.7) in page 52 of \cite{merle2006sharp}.

Now, for the purpose of our work, we write a corollary of Theorem \ref{mr1}.
\begin{cor}\label{mr1cor}
Assume  $u$ is a solution to \eqref{nls} with $H^{1}$ initial data, nonpositive  energy, and satisfying assumption \eqref{supercritical}, assume further $\|u\|_{2} \neq \|Q\|_{2}$ if $u$ is of zero energy, 
 then there exists $\delta=\delta(u)>0$, such that $\forall A>1$, there exists $T_{1}<T^{+}(u)$, $x_{1}\in \RRR^{d}, l_{1}>0$, such that
 \begin{equation}
 \int_{|x-x_{1}|\leq l_{1}} |u(T_{1},x)|^{2}\geq \delta, \quad  \int_{|x-x_{1}|\geq Al_{1}} |u(T_{1},x)|^{2}\geq \delta
 \end{equation}
\end{cor}
See proof in Section \ref{sectioncor}.

\section{An overview for the proof for Theorem \ref{main1}, Theorem\ref{main2}}
We give an overview for the proof of Theorem \ref{main1} , Theorem \ref{main2}. We mainly focus on the proof the Theorem \ref{main1}. Indeed, Theorem \ref{main2} follows from Theorem \ref{main1} due to the following classical fact:
\begin{lem}
Let $H$ be an Hilbert space, if $v_{n}\rightharpoonup v_{0}$ and $\|v_{n}\|_{H}=\|v\|_{H}$, then
\begin{equation}
v_{n}\rightarrow v_{0}.
\end{equation}
\end{lem} 
\subsection{Introduction}
Ever since the work of Kenig and Merle, \cite{kenig2006global}, \cite{kenig2008global}, there is a road map to approach results of type  Theorem \ref{main1}, which includes three ingredients:
\begin{enumerate}
\item concentration compactness theorems
\item variational characterization of ground state
\item rigidity theorems
\end{enumerate}
Concentration compactness relies on the study of the \textbf{linear} operator $e^{it\Delta}$ and remedies the lack of compactness of classical Strichartz estimate caused by the symmetry of the system.

 The concentration compactness will help us reduce the study of the original problem to the study of so-called almost periodic solution, i.e. solution of the form
\begin{equation}\label{almost}
\begin{aligned}
&u(t,x)=\frac{1}{\lambda^{d/2}(t)}P_{t}(\frac{x-x(t)}{\lambda(t)})e^{ix\xi(t)}, \lambda(t)>0, x(t),\xi(t)\in \RRR^{d},\\
& \{P_{t}\}_{t} \text{ precompact familily in } L^{2}(\RRR^{d}).
\end{aligned}
\end{equation}
Such strategy is a also called Liouville Theorem in the literature, see for example, \cite{merle2004universality}.

Variational characterization will help us understand  why ground state $Q$ is special. Thus , help us see the profile $Q$ in the study of \eqref{almost}.

Rigidity theorem will tell us the so called almost periodic solution is special, and one may expect a powerful enough rigidity theorem should fully characterize solutions to \eqref{nls} of type \eqref{almost}, though we cannot achieve this in this  work.
\subsection{Step 1: First extraction of profile}
First, we will use the profile decomposition and a minimization argument in \cite{dkmnew} to show the following:
\begin{lem}\label{firstextraction}
Let $u$ be a solution, not necessarily radial, satisfying  the assumption of Theorem \ref{main1}, then there exists a sequence $t_{n}\rightarrow T^{+}(u)$, such that $u(t_{n})$ admits a profile decomposition with profiles $\{\phi_{j}, \{x_{j,n},\lambda_{j,n},\xi_{j,n},t_{j,n}\}_{n}\}_{j}$, and there is a  unique compact profile, we assume it is  $\phi_{1}$, such that 
\begin{itemize}
\item $\|\phi_{1}\|_{2}\geq \|Q\|_{2}$,
\item The associated nonlinear profile $\Phi_{1}$, is an almost periodic solution in the sense of \eqref{almost}, and it does not scatter forward nor scatter backward.
\end{itemize}
\end{lem}
See Section \ref{sectionfirstextraction} for a proof.
\begin{rem}
One may compare this step to the procedure of reduction to the minimal blow up solution in the study of defocusing problem.
\end{rem}
\begin{rem}
Due to the assumption  \eqref{supercritical}, there cannot be  more than one profile with mass no less than $\|Q\|_{2}^{2}$.
\end{rem}
\subsection{Step 2: Second extraction of Profile}
We need to do some further modification of profile, the following step is very standard when one wants to prove scattering type results. By arguing exactly as Section 4 of \cite{tao2008minimal}, we will have
\begin{lem}\label{secondextraction}
Let $\Phi_{1}$ be the nonlinear profile as in Lemma \ref{firstextraction}, with lifespan $(T^{-},T^{+})$. Then, according to Lemma \ref{firstextraction}, 
\begin{equation}\label{almostsecond}
\Phi_{1}(t,x)=\frac{1}{\lambda_{1}^{d/2}(t)}P_{t}(\frac{x-x_{1}(t)}{\lambda_{1}(t)})e^{ix\xi_{1}(t)}.
\end{equation}
Further more, there exists $\{t_{n}\}_{n}, t_{n}\in (T^{-},T^{+})$, such that
\begin{equation}\label{sc}
P_{t_{n}}\xrightarrow{n\rightarrow \infty} P_{0}  \text{ in } L^{2}
\end{equation}
And  the solution $w$ to \eqref{nls} with initial data $P_{0}$ or $\bar{P_{0}}$ satisfies
\begin{equation}\label{minimalcouter}
w(t,x)=N^{d/2}(t)L_{t}(N(t)(x-x(t)))e^{ix\xi(t)}. t\geq 0, N(t)\leq 1
\end{equation}
And $\{L_{t}\}_{t}$ is a precompact $L^{2}$ family.
\end{lem}

Indeed,such a solution $w$, sometimes also called minimal blow up solution, already partially falls into the framework of Dodson's work \cite{dodson2012global}, \cite{dodson2015global}.

\subsection{Step 3: Fast Cascade case}
We  exclude the so-called fast cascade, i.e.  the case
\begin{equation}\label{fastcascade}
\int_{0}^{\infty}N^{3}(t)<\infty.
\end{equation}

In this regime, for $d=3$, Dodson's long time Strichartz estimate, Theorem 1.24 in \cite{dodson2012global} will indeed imply $w$ is not only a $L^{2}$ solution, but an $H^{1}$ solution, see Theorem 3.13 in \cite{dodson2012global}, and furthermore, the energy is zero, see (3.86), (3.87) and Remark 3.14 in \cite{dodson2012global}.  Long time Strichartz estimate also holds for $d=1,2$,  with extra technical difficulty, see for \cite{dodson2016global}, \cite{dodson2467global}. See Theorem 1.9 in \cite{dodson2015global} for a summary.
  
Thus, we have
\begin{lem}[Dodson]\label{extraregular}
Consider $w$ as in Lemma \ref{secondextraction}, assume further \eqref{fastcascade}, then $w(0)\in H^{1}$ and $E(w)=0$.
\end{lem}
We then have
\begin{lem}\label{nofastcascade}
Consider $w$ as in Lemma \ref{secondextraction}, \eqref{fastcascade} cannot hold.
\end{lem}
\begin{proof}
The case $\|w\|_{2}=\|Q\|_{2}$ is impossible since by Lemma \ref{extraregular}, $w$ is in $H^{1}$ and with zero energy, thus, by Lemma \ref{lemvarchar}, $w=\frac{1}{\lambda_{0}^{d/2}}Q(x-x_{0}/\lambda_{0})^{e^{i\gamma}}$ and the solution is just a standing wave which implies $N(t)\sim 1$ and $\int_{0}^{\infty} N(t)^{3}=\infty.$
The case $\|Q\|<\|u\|_{2}<\|Q\|_{2}+\alpha$ is impossible because by Theorem \ref{mr1}, such solution must blow up in finite time.
\end{proof}

\subsection{Step 4: Quasisoliton case}\label{subsectionqua}
It is in this step that we need radial assumption . Since we are considering radial solution, then $w$ in \eqref{minimalcouter} must also be radial, which imply that $x(t), \xi(t)\equiv 0$.
\begin{rem}
Since all we need is  $x(t), \quad \xi(t)\equiv 0$,   we may just assume $u$ is symmetric across $d$ linear independent planes. The observation that  if $u$ is symmetric across $d$ linear independent planes then  $x(t), \xi(t)\equiv 0$ has been pointed out by Dodson \cite{dodson2012global}.
\end{rem}
Now, we are left with the case $\int_{0}^{\infty} N^{3}(t)=\infty$, which is usually called Quasoliton case in the literature. We will show in this case,  it must be that $\|w(0)\|_{2}=\|Q\|_{2}$.
\begin{lem}\label{quasuper}
It is impossible that $w$ is of  form \eqref{minimalcouter}, $x(t),\xi(t)\equiv 0$, $\|Q\|_{2}<\|w(0)\|_{2}\leq \|Q\|_{2}+\alpha$, and $\int_{0}^{\infty} N^{3}(t)=\infty$.
\end{lem}
And we will further show 
\begin{lem}\label{quathresh}
Assume $w$ is  of form \eqref{minimalcouter}, $x(t),\xi(t)\equiv 0$, $\|w(0)\|_{2}=\|Q\|_{2}$, and $\int_{0}^{\infty} N^{3}(t)=\infty$, then exist a sequence $t_{n}$, and parameters $\lambda_{n},  \gamma_{n}$ such that
\begin{equation}\label{strongconvergence}
\lim_{n\rightarrow \infty}\|\lambda_{n}^{d/2}w(t_{n}, \lambda_{n}x)e^{i\gamma_{n}}-Q\|_{2}=0.
\end{equation}
\end{lem}
\begin{rem}
It is very natural to conjecture that the under the same assumption of Lemma \ref{quathresh}, $w$ is essentially standing wave $Qe^{it}$. This will be related the classification of  finite time blow up solution to \eqref{nls} with mass $\|u_{0}\|_{2}=\|Q\|_{2}$. Such solutions, if one further assume the initial data is in $H^{1}$, are completely determined by the result of Merle, \cite{merle1993determination}. At the level of $L^{2}$, it seems to be a very hard problem.
\end{rem}

To understand the proof of Lemma \ref{quasuper}, Lemma \ref{quathresh}, one needs to understand how Dodson handles the case $\|w(0)\|_{2}<\|Q\|_{2}$, \cite{dodson2015global}. We will give a rather detailed sketch of Dodson's arguments in Section \ref{quasisolitoncase}. Basically, one needs to use Virial identity to explore the decay of the solution and one needs to perform frequency cut-off to explore the coersiveness of energy. We will show the following:
\begin{lem}\label{seeyou}
Assume $w$ is  of form \eqref{minimalcouter}, $x(t),\xi(t)\equiv 0$, and $\int_{0}^{\infty} N(t)^{3}=\infty$, then
there exist  sequences $t_{n}\leq T_{n}$, $R_{n}\gg \frac{1}{N(t_{n})} $, 
$\int_{0}^{T_{n}}N^{3}(t)=K_{n}$,  such that
\begin{equation}
E\left(\chi(\frac{x}{R_{n}})P_{\leq CK_{n}}w(t_{n})\right)\leq \frac{1}{n} \left\|\nabla \chi(\frac{x}{R_{n}})P_{\leq CK_{n}}w(t_{n}))\right\|_{2}^{2},
\end{equation}
where $\chi$ is smooth bump function localized around the origin.
\end{lem}

We will use  Lemma \ref{seeyou},  Lemma \ref{lemvarchar} and Corollary \ref{mr1cor} to deduce Lemma \ref{quasuper}, Lemma \ref{quathresh}.

See Section \ref{quasisolitoncase} for the proof of Lemma \ref{seeyou}, Lemma \ref{quasuper} and Lemma \ref{quathresh}.
\subsection{Step 5: Approximation argument and conclusion of the proof}\label{finalapproximation}
To conclude the proof of Theorem \ref{main1},  we use Lemma \ref{firstextraction} to reduce the dynamics of $u$ to the unique compact profile $\phi_{1}$, and its associated solution $\Phi_{1}$. Then we use Lemma \ref{secondextraction} to reduce the dynamic of $\Phi_{1}$ to the the almost periodic solution $w$, and  we use Lemma \ref{nofastcascade},Lemma \ref{quasuper} to derive that 
one must have
\begin{equation}
\|w\|_{2}=\|Q\|_{2},\quad  \int_{0}^{\infty} N^{3}(t)=\infty.
\end{equation}
And finally, such solution is characterized by Lemma \ref{quathresh}. We show the detail in Section \ref{step5}.

\section{Proof of Corollary  \ref{mr1cor}}\label{sectioncor}
We prove Corollary \ref{mr1cor} here. Let $u$ be as in Corollary \ref{mr1cor}.

First, if $u$ is of  strictly negative energy,   by  Lemma \ref{lemvarchar}, we have
\begin{equation}\label{strictabove}
\|u\|_{2}>\|Q\|_{2}
\end{equation}
If $u$ is of zero energy, then \eqref{strictabove} is already in the assumption of Corollary \ref{mr1cor}.

Thus, we can assume
\begin{equation}\label{masscondition}
 \|u\|_{2}^{2}\equiv \|Q+\epsilon\|_{2}^{2}= \|Q\|^{2}_{2}+\delta_{0}.
 \end{equation}
 Note mass is a conservation law.
  By choosing $\alpha $ in Assumption \ref{supercritical}, $\delta_{0}\ll \int_{|x|\leq 1}|Q|^{2}dx$ We will choose the $\delta(u)$ in Corollary \ref{mr1cor} as $\frac{\delta_{0}}{2}$.

Since $Q$ is  of exponential decay, we have  that, by \eqref{remainder}, when $t$ is close to $T^{+}(u)$ enough,
\begin{equation}\label{1111}
<|Q|, |\epsilon(t,x)|>\ll \delta_{0}.
\end{equation}
By \eqref{masscondition},  we obtain
\begin{equation}\label{rmaindermasscondition}
 \int |\epsilon|^{2}\lesssim \frac{3}{4}\delta_{0}.
 \end{equation}

On the other hand, by \eqref{1111}
\begin{equation}
\int_{|x|\leq 1}{|Q+\epsilon|^{2}}\geq \delta_{0}/2.
\end{equation}

Now fix any $A>1$, using the trivial estimate
\begin{equation}
\int_{|x|\leq A} |\epsilon|^{2}|\lesssim e^{A}\int |\epsilon|^{2}e^{-|x|}
\end{equation}

By \eqref{remainder}, and  $T_{1}$ close to $T^{+}(u)$, we have 
\begin{equation}\label{2222}
\int_{|x|\leq A}|\epsilon|^{2}\ll \delta_{0}
\end{equation}

Thus combine \eqref{1111}, \eqref{2222} and \eqref{rmaindermasscondition}, we have by triangle inequality that
\begin{equation}
\int_{|x|\geq A} |Q+\epsilon|^{2}\geq \delta_{0}/2.
\end{equation}
Let $l_{0}, x_{0}$ in Lemma \ref{mr1cor} be $x(T_{1}), \lambda(T_{1})$, then the Corollary follows.
\section{Proof of Lemma \ref{firstextraction}}\label{sectionfirstextraction}
Lemma \ref{firstextraction} should be compared with the reduction to minimal mass blow up solutions for scattering type problem. Most arguments below are standard in concentration compactness, see for example \cite{kenig2006global}, \cite{kenig2008global},  thus we just sketch it. We will also use a minimization procedure from \cite{dkmnew}, which makes the whole proof more clear for us. We remark that we do not use the fact that $u$ is radial here.

First, for any $\tilde{t}_{n}\rightarrow T^{+}$, up to extracting subsequence, we may assume $u(\tilde{t}_{n})$ admits profile decomposition with profiles $\{\psi_{j},\{\tilde{x}_{j,n}, \tilde{\xi}_{j,n}, \tilde{\lambda}_{j,n}, \tilde{t}_{j,n}\}_{n}\}_{j}$. If $\forall j$, we have $\|\psi_{j}\|_{2}<\|Q\|_{2}$, then by  Theorem \ref{dodson}, and the nonlinear approximation argument  Proposition \ref{nonlinearprofile}, one would derive that $u$ scatters forward, which contradicts our assumption. Thus, there is at least one profile with mass no less than $\|Q\|_{2}^{2}$. On the other hand, by the asymptotically orthogonality of the mass, \eqref{orthmass}, and our assumption \ref{supercritical}, there can only be one profile with mass no less than $\|Q\|_{2}$. By reordering the profile if necessary, we assume the first profile $\psi_{1}$ is the unique profile with 
\begin{equation}\label{abovegroundstate}
\|\psi_{1}\|_{2}\geq \|Q\|_{2}.
\end{equation}
\begin{rem}
By Assumption \ref{supercritical} and the asymptotically  orthogonality of mass, all other profiles has mass $\ll 1$, which implies there associated nonlinear profile is global and scattering by the small data theory, Proposition \ref{smalldata}.
\end{rem}
Furthermore, $\psi_{1}$ must be a compact profile. Indeed, if $\psi_{1}$ is a forward scattering profile,  using the nonlinear approximation Proposition \ref{nonlinearprofile}, we have that $u$ scatter forward, a contradiction. If $\phi_{1}$ is a backward scattering profile, then we use Propositon \ref{nonlinearprofile} for the initial data $u(\tilde{t}_{n})$, but run the \eqref{nls} backwards rather than forwards, then we will get a uniform bound for $\|u\|_{L^{2(d+2)/d}_{t,x}[0,t_{n}\times \RRR^{d}]}$, which again implies $u$ scatters forward, a contradiction. See \cite{kenig2006global} for similar arguments for energy critical wave, see also \cite{kenig2008global}. Since $\psi_{1}$ is a compact profile, we do not distinguish  between $\psi_{1}$ and  its associated nonlinear profile $\Psi_{1}$, which is the solution to \eqref{nls} with initial data $\psi_{1}$.

Finally, we remark $\psi_{1}$ is not uniquely determined by the times sequence $\{\tilde{t_{n}}\}_{n}$, since one may scale or translate the profile, however, the $L^{2}$ norm of $\psi_{1}$ is uniquely determined, since $L^{2}$ is invariant under those symmetry .

To find a sequence $\{t_{n}\}$ such that its associated profile decomposition satisfy Lemma \ref{firstextraction}, we mimic the minimization procedure in Section 4 of \cite{dkmnew}. Though that paper deals with energy critical wave and energy critical Schr\"odinger, most arguments there  are quite general and works whenever there is a satisfying profile decomposition technique.  

One will need  the so-called  double profile decomposition at the technique level, which maybe compared to the diagonal  technique which is used in the proof Arzela-Ascoli Lemma.
\begin{lem}\label{double}
Assume $\{f_{n}^{p}\}_{n,p}$ are uniformaly bounded in $L^{2}$, assume for $\forall p$, $f_{n}^{p}$ admits a profile decomposition with profiles $\{g_{j}^{p}\}$ such that there exits $\{\eta_{j}\}_{j}$ such that for all $p$,
\begin{equation}\label{techdouble}
\sum_{j}\eta_{j}<\infty, \quad  \|e^{it\Delta}g_{j}^{p}\|\leq \eta_{j}.
\end{equation}
And assume for all $j$, $\{e^{it\Delta}g^{p}_{j}(0)\}_{p}$ admits a profile decomposition with profile $h_{j,k}$, then up to extracting subsequence, there exists $n_{p}\rightarrow \infty$ such that $\{f_{n_{p}}^{p}\}_{p}$ admits a profile decomostion with profile $\{h_{j,k}\}_{j,k}$.
\end{lem}
\begin{rem}
According to asymptotically orthogonality of mass \eqref{orthmass}, we have that for all $j,k$, $\|h_{j,k}\|_{2}\leq \|g_{j}\|_{2}$. 
\end{rem}
\begin{rem}
We will not need to check condition \eqref{techdouble} in our work. because for all the profile decompositions involved in our work, if we reorder the profiles such that $\|g_{j}\|\geq \|g_{j'}\| , \forall j\geq j'$, we always have $\|g_{1}\|_{2}\leq  \sqrt{2\|Q\|_{2}}$, and $\|g_{n}\|_{2}\lesssim \sqrt{\frac{1}{n-1}}$, $\forall n\geq 2$, thus \eqref{techdouble} automatically holds.
\end{rem}
Lemma \ref{double} is the natural generalization of Lemma 3.16 in \cite{dkmnew} for equation \eqref{nls}. we refer to \cite{dkmnew} for a proof. (Though the proof there is written for energy critical wave, it also works here.)

Now let us go back to the proof of Lemma \ref{firstextraction}.

Let $\AAA$ be the set of the time sequence $\{\tilde{t}_{n}\}_{n}$ such that  $\lim_{n\rightarrow}\tilde{t}_{n}=T^{+}(u)$ and $\{u(\tilde{t}_{n})\}_{n}$ admits a profile decomposition. Let $\phi_{1}$ be the profile with mass no less than $\|Q\|_{2}^{2}$. Recall that $\phi_{1}$ may not be uniquely determined  by the time sequence, but $\|\phi_{1}\|_{2}$ is . We define a map for $s=\{\tilde{t}_{n}\}_{n} \in A$ to $\RRR$ as 
\begin{equation}
\EEE(s)=\|\phi_{1}\|_{2}=\|\Psi_{1}\|_{2}
\end{equation}

We have that \eqref{abovegroundstate} implies 
\begin{equation}
\inf_{s\in\AAA}\EEE\geq \|Q\|_{2}.
\end{equation}
We now claim there exists an $s_{0}\in \AAA$ such that
\begin{equation}\label{miniexists}
\EEE(s_{0})=\inf_{s\in\AAA}\EEE.
\end{equation}
In fact, by Lemma \ref{double}, there exists $\EEE({s_{p}})\rightarrow \inf_{s\in\AAA}\EEE$, then apply the double profile decomposition Lemma \ref{double}, one will find $ \{u(\tilde{t}_{n_{p}}^{p})\}$ admits a profile decompostion and $ \EEE(\{\tilde{t}_{n_{p}^{p}}\}_{p})=\inf_{s\in\AAA}\EEE.$

Let us  finish the proof of Lemma \ref{firstextraction}.
Let $s_{0}=\{t_{n}\}_{n}$,  and $u(t_{n})$ admits a profile decomposition  with profiles $\{\phi_{j},\{x_{j,n}, \xi_{j,n}, \lambda_{j,n}, t_{j,n}\}_{n}\}_{j}$ , the associated unique nonlinear profile with mass above ground state be $\Phi_{1}$, clearly $\Phi_{1}$ satisfy
\begin{equation}
\|Q\|_{2}\leq \|\Phi_{1}\|_{2}\leq \|Q\|_{2}+\alpha.
\end{equation} 
We further claim $\Phi_{1}$ must be an almost periodic of form \eqref{almost}. Indeed, if not, then there exists time sequence $\{a_{n}\}$ within the life span of $\Phi_{1}$ such that $\Phi_{1}(a_{n})$ admits a profile decomposition and  any profile has mass strictly smaller than $\|\Phi_{1}\|_{2}$, then using the nonlinear approximation Proposition \ref{nonlinearprofile} and double profile decomposition Lemma \ref{double}, it is easy to see up to picking a subsequence ,$ u(t_{n}+\lambda^{2}_{1}a_{n})$ that admits a profile decompostion and  $\EEE(\{t_{n}+\lambda^{2}_{1}a_{n}\}_{n})<\EEE(s_{0})$, a contradiction. This concludes the proof.

\section{Proof for Subsection \ref{subsectionqua}}\label{quasisolitoncase}
We prove Lemma \ref{seeyou}, Lemma \ref{quasuper}, Lemma \ref{quathresh} here.

To prove Lemma \ref{seeyou}, one needs to use  the proof in \cite{dodson2015global}. We do a review of the argument in \cite{dodson2015global} here. 
\subsection{A quick review of Dodson's work \cite{dodson2015global}}
\subsubsection{warm up and energy tensor}
One is recommended to use the associated energy tensor to do computation. We use Einstein summation convention.

Let
\begin{equation}
iu_{t}+\Delta u=-|u|^{p-1}u.
\end{equation}
(Note in our case, $p=1+4/d$.)

and  let

\begin{equation}\label{tensor}
\begin{aligned}
&T_{00}=T_{00}(u):=|u|^{2},  \quad T_{j0}=T_{0j}=2\Im u_{j}\bar{u},\\
 &T_{kj}=T_{jk}=4\Re \bar{u}_{j}u_{k}-\delta_{kj}\Delta |u|^{2}-2\frac{p-1}{p+1}|u|^{p+1},
\end{aligned}
\end{equation}

Then, 
\begin{equation}
\partial_{t}T_{00}+\partial_{j}T_{0j}=0, \partial_{t}T_{0j}+\partial_{k}T_{jk}=0.
\end{equation}

Let us recall Virial identity as an example.
The computation here is just formal, we assume a priori all the quantities in the following computation is finite.
\begin{equation}\label{virial}
\begin{aligned}
  &\partial_{tt}\int |x|^{2}|u|^{2}
=\partial_{tt}\int |x|^{2}T_{00}
=\partial_{t}\int |x|^{2}\partial_{t}T_{00}\\
=&\partial_{t}\int \partial_{j} |x|^{2}T_{j0}
=\int \partial_{j}\partial_{k}|x|^{2}T_{j,k}=16E(u).
\end{aligned}
\end{equation}

Almost all  results regarding the long time dynamic of \eqref{nls} rely on \eqref{virial} in some sense. Intuitively, when $E(u)$ is positive, (which is always the case when $\|u\|_{2}<\|Q\|_{2}$ ), then $\int |x|^{2}|u|^{2}$ will  grow to infinity. On the other hand, since the mass $\int |u|^{2}$ is conserved, this implies that mass are ejected to infinity, which should be understood as a dispersion effect.

We further summarize the last three identities in \eqref{virial}.
\begin{equation}\label{v1}
\frac{d}{dt}\int x_{j}T_{0j}\equiv \frac{d}{dt}\int x_{j}\Im u_{j}\bar{u}=\frac{1}{2}\int \partial_{k}x_{j}T_{jk}\equiv4E(u).
\end{equation}
\subsubsection{A sketch of Dodson's work}
Using \eqref{v1} , Dodson shows
\begin{prop}\label{np}
it is impossible that \eqref{minimalcouter} holds with  $\int N^{3}(t)=\infty$ and 
\begin{equation}\label{subcrticl}
\|w\|_{2}<\|Q\|_{2}-\eta,  \text{ for some } \eta>0
\end{equation}
unless $w\equiv 0$
\end{prop}
 We will use $d=3$ here to present a sketch for the proof of Proposition \ref{np}. We only do the case $\xi(t), x(t)=0$, (recall $\xi(t), x(t), N(t)$ in \eqref{minimalcouter} ).
 \footnote{ Indeed, the computation is easier for $d=1$ or $d=2$, however, Dodson's work \cite{dodson2015global} implicitly use his long time Strichartz estimate, which involves extra technical difficulty for $d=1,2$. }.

\textbf{First subcase: $\xi(t)\equiv 0,x(t)\equiv 0, N(t)\sim 1$ }

Let $T_{K}$ be the unique time such that
\begin{equation}
\int_{0}^{T_{K}}N^{3}(t)=K.
\end{equation}

All the analysis below is in $[0,T_{K}]$ for $w$ as in \eqref{minimalcouter}. 

We first consider the subcase $\xi t)\equiv 0,  x(t)\equiv 0,  N(t)\sim 1$ for $w$ in \eqref{minimalcouter}.  We remark that  $x(t), \xi(t)\equiv 0$ when one only considers radial solutions.  Now the solution is like a soliton, without dispersion. It is very natural to apply \eqref{v1} to get a contradiction. However, $w$ is not in $H^{1}$, so one cannot directly use mass constriction \eqref{subcrticl} to use the coersiveness of energy and one does not have the extra integrability of $xu$ to make sense the left side of \eqref{v1}. So, very naturally, one needs to do truncation in space and frequency.

We need a cut off version of $x$, for technical reason, we will need a $\psi(x)$ such that $\psi(x)=1$ for $|x|\leq 1$, and $\psi(x)\lesssim \frac{1}{|x|}$, and $\partial_{k}\psi(x)x_{j}$ is semi positive definite \footnote{ This  is not hard, indeed, one first constructs some convex $f(x)$ which is like $|x|^{2}$ near the origin, slowly grows for $|x|\geq 2$ ,and take $\psi(x) x_{j}=\partial_{j}f(x)$. Note we do not need $f$ to be uniformly strictly convex.}.

We also define the Fourier truncation, $I:=I_{K}\equiv P_{\leq CK}$, here $C$ is some fixed large constant. Let $F(v):=-|v|^{4/d}v$. Now note 
\begin{equation}\label{mnls}
(i\partial_{t}+\Delta)Iw=I(F(w))=F(I(w))+\{I(F(w)-F(Iw)\}.
\end{equation}
Let $M(t)$ be the truncated version of $\int x_{j}\Im w_{j}\bar{w}$, ($M$ denotes Morawetz action in the literature) :
\begin{equation}\label{moraaction}
M(t)=\int \psi(\frac{x}{R})x_{j}\Im I_{K}w_{j}I_{K}\bar{w}.
\end{equation}
 Since $\int |w|^{2}$ is bounded, one immediately obtain that
\begin{equation}
|M(t)|\lesssim RK.
\end{equation}
Recall \eqref{minimalcouter},  using the fact that $w=N(t)^{d/2}L_{t}(N(t)x)$, $N(t)\sim 1$ and $\{L_{t}\}$ is a precompact in $L^{2}$, it is easy to upgrade the above to
\begin{equation}\label{moraestimate1}
|M(t)|\lesssim Ro(K).
\end{equation} 

Now computing the derivative of $M(t)$, one has
\begin{equation}
\begin{aligned}
\frac{d}{dt}M(t)=&E_{1}+\int \partial_{k}(\psi(x/R))x_{j}\frac{1}{2}T_{jk}(Iw)\\
=&E_{1}+\int \partial_{j}\partial_{k}(\psi(x/R)x_{j})\frac{1}{2}\{4\Re \bar{Iw}_{j}Iw_{k}-\delta_{kj}\Delta |Iw|^{2}-\frac{2\cdot 4/d}{4/d+2}|w|^{\frac{4}{d}+2}\}\\
=&E_{1}+E_{2}+E_{3}+4\int_{|x|\leq R}\left (\frac{1}{2}|\nabla u|^{2}-\frac{1}{2+\frac{4}{d}}|w|^{2+4/d}\right),
\end{aligned}
\end{equation}

 shere $E_{1},E_{2},E_{3}$ are as the following.
\begin{eqnarray}\label{threeerror}
&&E_{1}=-i\int \psi (\frac{x}{R})x\left(\{IF(w)-F(Iw)\}\nabla Iw+Iw \nabla  \{IF(w)-F(Iw)\}\right),\\
&&E_{2}=\int [\Delta \delta_{jk}\psi(x/R)x_{j}](-\delta_{kj}|Iw|^{2}),\\
&&E_{3}= \int_{|x|\geq R} \partial_{k}(\psi(x/R)x_{j})\frac{1}{2}\{4\Re \bar{Iw}_{j}Iw_{k}-\frac{2\cdot 4/d}{4/d+2}|w|^{\frac{4}{d}+2}\}.
\end{eqnarray}
Here $E_{1}$ is caused by the commutator type error in \eqref{mnls}.

One needs to explore the coersiveness of $\int_{|x|\leq R}\frac{1}{2}|\nabla Iw|^{2}-\frac{1}{2+\frac{4}{d} }|Iw|^{6}$, thus one needs to introduce an extra smooth cut-off function $\chi(x)$ which is 1 for $|x|\leq \frac{9}{10}$, and vanishes for $|x|\geq 1$. Then
\begin{equation}\label{coersivenss}
\int_{|x|\leq R}\frac{1}{2}|\nabla Iw|^{2}-\frac{1}{2+\frac{4}{d}}|Iw|^{6}\geq E(\chi(x/R)Iw)
\end{equation}
We remark that strictly speaking, \eqref{coersivenss} is not completely right, since we neglect the error caused by the commutator $(\nabla (\chi Iw)-\chi \nabla Iw)$, since this is just a sketch, we omit this technical point, one should refer to \cite{dodson2015global}  for more details.

Now, using \eqref{sharpGN} and the important condition \eqref{subcrticl}, one has
\begin{equation}\label{keycoer}
E(\|\chi(\frac{x}{R})Iw\|)\geq c_{0}(\eta)(\|\chi(\frac{x}{R})Iw\|^{4/d+2}+\|\nabla (\chi(x/R)Iw)\|^{2}).
\end{equation}
Error $E_{2}$, $E_{3}$ will be estimated by 
\begin{equation}\label{e2}
\begin{aligned}
|E_{2}|\lesssim \frac{1}{R^{2}}\int_{|x|\geq R}|Iw|^{2},\\
|E_{3}-\int_{|x|\geq R} \partial_{k}(\psi(x/R)x_{j})\frac{1}{2}\{4\Re \bar{Iw}_{j}Iw_{k}|\lesssim  \int_{|x|\geq R}|Iw|^{2+4/d}
\end{aligned}
\end{equation}
Note since $\partial_{k}(\psi x_{j})$ is semipositive definite, we have
\begin{equation}\label{e3}
 E_{3}\geq -C_{1}\int_{|x|\geq R}|Iw|^{4/d+2}  \text{ for some } C_{1}>0.
\end{equation}
$E_{1}$ is estimated by the commutator type estimate,  see Lemma 4.7 in \cite{dodson2012global}.

\begin{equation}\label{finalerror}
\|I_{K}F(w)-F(I_{K} w)\|_{L_{t}^{2}L_{x}^{2d/d+2}[0,T_{K}]}\lesssim o_{K}(1).
\end{equation}
The proof of the above relies on Dodson's long time Strichartz estimate, Theorem 1.24 in \cite{dodson2012global} , which is indeed the key ingredient in \cite{dodson2012global}.

Thus,
\begin{equation}
\left|\int_{0}^{T_{K}}E_{1}\right|\lesssim Ro_{K}(1)\|\nabla I_{K}w\|_{L^{2}L_{x}^{2d/d-2}}.
\end{equation}
And the long time Strichartz estimate will further give, see Lemma 4.5 in \cite{dodson2012global},
\begin{equation}\label{lll}
\|\nabla I_{K}w\|_{L^{2}L_{x}^{2d/d-2}}\lesssim  K
\end{equation}
Thus, 
\begin{equation}\label{finalerror2}
\int_{0}^{T_{K}}|E_{1}|\geq Ro_{K}(1)K.
\end{equation}

We remark that  the long time Strichartz estimate is purely analytic, does not relies on the energy structure, i.e. the difference between focusing and defocusing does not matter here.

To summarize,
\begin{equation}\label{endofmora}
\begin{aligned}
\frac{d}{dt}M(t)&\geq 4E(\chi(x/R)Iw)+E_{1}+E_{2}+E_{3}\\
&\geq c_{0}(\eta)\|\chi(x/R)Iw\|_{L^{4/d+2}}+E_{1}-|E_{2}|-C_{1}\int_{|x|\geq R} |Iw|^{4/d+2}
\end{aligned}
\end{equation} 
In the last step, we pluged in the estimate \eqref{e3}.

Now, integrate in time  on $[0,T_{k}]$ plug in the estimate for $E_{1}$, \eqref{finalerror2}, and estimate for $E_{2}$, \eqref{e2}, we recover the estimate (3.26) \cite{dodson2015global}\footnote{the numerics here are slightly different, because in  \cite{dodson2015global}, that part is done for $d=1$, and the error caused by $E_{1}$ is neglected in this step but treated later.}.

\begin{equation}\label{conclumora}
\int_{0}^{T_{K}}\frac{d}{dt}M(t)\geq \int_{0}^{T_{k}}c_{0}(\eta)\chi(x/R)Iw\|_{L_{x}^{2+4/d}}^{2+4/d}-\int_{0}^{T_{K}} \sup_{t}\frac{1}{R^{2}}\|Iw(t)\|_{L^{2}_{x}}^{2}-\int_{0}^{T_K}\int_{|x|\geq R}Iw(t,x)|^{6}-Ro_{K}(1)K.
\end{equation}
Estimate \eqref{conclumora} is enough to give a contradiction and  conclude $u$ must be zero, one is refer to  \cite{dodson2015global}, in particular (3.26) in \cite{dodson2015global} for more details. We sketch some standard arguments  below for the convenience of the readers.

 The left side of \eqref{conclumora}  is controlled by Ro(K), by estimate\eqref{moraestimate1}.

On the other hand, since we assume $N(t)\sim 1$, 
$\int_{0}^{T_{K}}\sim \int_{0}^{T_{K}}N(t)^{3}=K$, thus 
\begin{equation}
\int_{0}^{T_{K}} sup_{t}\frac{1}{R^{2}}\|u(t) \|_{L^{2}_{x}}^{2}\lesssim \frac{1}{R^{2}}K
\end{equation}
And, recall again $u$ is of form \eqref{minimalcouter}, $N(t)\sim 1$, then for any time interval $J$ of length $\sim 1$, local theory of  \eqref{nls} gives $\int_{J}|u|_{L_{x}^{4/d+2}}^{4/d+2}\sim 1$. Now using the fact $u$ is of form \eqref{conclumora}, i.e. all the mass of $u$ is uniformly concentrated in physical space and frequency space, we  obtain
\begin{equation}
\int_{J}|\chi(x/R)Iw^{4/d+2}_{L_{x}^{4/d+2}}\sim 1, \int_{J}\int_{|x|\geq R} \lesssim o_{R}(1).
\end{equation}

Thus, the right side of \eqref{conclumora} is bounded below by
\begin{equation}
c_{0}(\eta)K-\tilde{C}\frac{K}{R^{2}}-o_{R}(K)-Ro(K).
\end{equation}
Here $\tilde{C}$ is some universal constant.
Now one obtains that $ Ro(K)\geq c_{0}(\eta)K-\tilde{C}\frac{K}{R^{2}}-o_{R}(K)-Ro(K)$, a contradiction.

The key point to conclude a contradiction by using $Ro(K)\geq c_{0}(\eta)K-\tilde{C}\frac{K}{R^{2}}-o_{R}(K)-Ro(K)$ is that here $c_{0}(\eta)$ in \eqref{keycoer} does not depend on $R$ or $K$, so one can first choose $R$ large enough, then one further choose $K$ large enough to get a contradiction.

\textbf{General Case: $\xi(t), x(t)\equiv 0$}

Now, for general case with $\xi(t), x(t)=0$ , (which covers all general radial solution), it is very natural to define the Morawetz action as 
\begin{equation}\label{gemora}
M(t):=\int \psi \left(\frac{x{N}(t)}{R}\right)x{N}(t)\Im \nabla Iw I\bar{w}.
\end{equation}
However, if one directly relies on  the above Morawetz action to argue as previously, one will facethe  problem that one does not have good control about $N'(t)$.  This is  handled by Dodson using so called "upcoming algorithm", basically he constructs some slowly oscillating $\tilde{N}(t)\lesssim N(t)$ according to the behavior of $N(t)$, and constructs Morawetz action as
\begin{equation}
M(t)=\int \psi \left(\frac{x\tilde{N}(t)}{R}\right)x\tilde{N}(t))\Im \nabla Iw I\bar{w}
\end{equation} 
The proof left follows is in principle as the previous subcase where $N(t)\sim 1$, see section 4 in \cite{dodson2015global} for more details.

As emphasized in the end of the previous case,  the key part and the only part Dodson's proof using the fact $\|u\|_{2}<\|Q\|_{2}-\eta$ is that this will gives a universal constant $c_{0}(\eta)$ such that
\begin{equation}\label{coingeneral}
E(\chi(\tilde{N}(t)x/R)Iw)\geq c_{0}(\eta )[\|\nabla \chi(\tilde{N}(t)x/R)Iw \|_{2}^{2}+\|\chi(\tilde{N}(t)x/R)Iw\|_{L^{4/d+2}_{x}}^{4/d+2}].
\end{equation}

\subsection{Proof of Lemma \ref{seeyou}}
The proof of Lemma \ref{seeyou} is by contradiction. Indeed, if Lemma \ref{seeyou} does not hold, then one recovers \eqref{coingeneral} even   the mass of $w$ is not under the ground state. Then, one argue as the proof of Proposition \ref{np}, which we just reviewed in the previous subsection, to conclude $w=0$, which is clearly a contridction since $\|w\|_{2}\geq \|Q\|_{2}.$

\begin{rem}
Proposition \ref{np} holds for general non radial solution by using a  version of interaction Morawetz estimate.  Ever since \cite{colliander2008global}, there are a lot of works using interaction version of certain estimates to show results for general solutions rather than radial solution, such as \cite{dodson2012global}, \cite{dodson2015global} and many others. However,we cannot have a natural useful generalization  of Lemma \ref{seeyou} here.  It seems to us if one directly follows the arguments in \cite{dodson2015global}, where a nonradial version Propostion \ref{np} is proved, one can only conclude that there exists a sequence $t_{n}\leq T_{n}$, $R_{n}\gg \frac{1}{N(t_{n})}, \xi_{n}, x_{n}$, 
$\int_{0}^{T_{n}}N^{3}(t)=K_{n}$, such that
\begin{equation}\label{fail}
E(\chi(\frac{x-x_{n}}{R_{n}})P_{\leq CK_{n}}e^{-i\xi_{n}x}w(t_{n}))\leq \frac{1}{n} \|\nabla \psi(\frac{x-x_{n}}{R_{n}})P_{\leq CK_{n}}e^{-i\xi_{n}x}w(t_{n}))\|_{2}^{2}
\end{equation}
Formula \eqref{fail} is  of no use to us, because we do not have control of $x_{n}$ here, thus we are not ensured $\chi(\frac{x-x_{n}}{R_{n}})P_{\leq CK_{n}}e^{-i\xi_{n}x}w(t_{n}))$ contains almost all the mass of $w$ as $n\rightarrow \infty$, which will be critical later.
\end{rem}

\subsection{Proof of Lemma \ref{quathresh}}
Lemma \ref{quathresh} is implied by Lemma \ref{seeyou}, with the help of Lemma \ref{mrvar}. We present a short proof here.  Recall $w$ is of form \eqref{minimalcouter}, and since we are consider radial solution,  we have $\xi(t)=0, x(t)=0$. Since $R_{n}\gg \frac{1}{N(t_{n})}$, $K_{n}\gg 1$ and $\{L_{t}\}_{t}$ is a precompact $L^{2}$ family, we easily have
\begin{equation}
\|w(t_{n})-\chi(x/R_{n})P_{<CK_{n}}w(t_{n})\|_{L^{2}}\xrightarrow{n\rightarrow \infty}0.
\end{equation}
Now let $\tilde{w}_{n}:=\chi(x/R_{n})P_{<CK_{n}}w(t_{n})$. Clearly,  to prove Lemma \ref{quathresh}, we only need to show there exists  $\lambda_{n}, \gamma_{n}$ such that
\begin{equation}\label{immecon}
\lambda_{n}^{d/2}w_{n}(\lambda_{n}x)e^{i\gamma_{n}}\rightarrow Q \text{  in } L^{2}.
\end{equation}

But \eqref{immecon} follows from Lemma \ref{mrvar} because
\begin{equation}
\begin{aligned}
\|w_{n}\|\leq \|w\|_{2}\leq \|Q\|_{2}+\frac{1}{n},\\
E(w_{n})\leq \frac{1}{n}\|\nabla w_{n}\|_{2}^{2}
\end{aligned}
\end{equation}
\subsection{Proof of Lemma \ref{quasuper}}
Now we turn to the proof of Lemma \ref{quasuper}. Note we restrict ourselves to radial solutions. Let $t_{n}, K_{n}, R_{n}$ be as in Lemma \ref{seeyou}, let $v_{n}=\chi(x/R_{n})P_{\leq CK_{n}}u(t_{n}), \lambda_{n}=\frac{\|\nabla Q\|_{2}}{\|\nabla v_{n}\|_{2}}$. We apriori have
\begin{equation}
\|v_{n}\|_{2}\leq \|Q\|_{2}+\alpha,\quad  \alpha \text{ small enough },
\end{equation}

Thus, when $n$ is large enough, $1/n\leq \alpha$, thus by Lemma \ref{mrvar}, we can find a sequence of $\gamma_{n}$, such that
\begin{equation}\label{perground}
\|\frac{1}{\lambda_{n}^{d/2}}v_{n}(\frac{x}{\lambda_{n}})e^{i\gamma_{n}}-Q\|_{H^{1}}\leq \delta(\alpha)\leq 1.
\end{equation}

Note the space translation parameter in Lemma \ref{mrvar} will not  appear because our functions are all radial.

Also recall $w(t)=N(t)^{d/2}L_{t}(N(t)x)$, $L_{t}$ is a precompact $L^{2}$ family, $N(t)\leq 1$,thus  the condition $K_{n}\gg 1, R_{n}\gg \frac{1}{N(t_{n})}$ implies 
\begin{equation}\label{uniformclo}
\|w(t_{n})-v_{n}\|_{2}=o_{n}(1)
\end{equation}
and further it  implies that $\{\frac{1}{N^{d/2}(t_{n})}v_{n}(\frac{x}{N(t_{n})})\}_{n}$ is a precompact $L^{2}$ family.

Thus, by  \eqref{perground} we have $N(t_{n})\sim \lambda_{n}$.
Now let $f_{n}=\{\frac{1}{N^{d/2}(t_{n})}v_{n}(\frac{x}{N(t_{n})})\}_{n}$, we have that
\begin{enumerate}
\item $f_{n}$ is a precompact $L^{2}$ sequence.
\item $f_{n}$ is uniformly bounded.
\item $E(f_{n})\lesssim \frac{1}{n}$. (This contains the possibility that $E(f_{n})$ is negative).
\end{enumerate}
Up to extracting a subsequence, we may assume $f_{n}$ strongly converges to $f_{0}$ in $L^{2}$, in particular, with \eqref{uniformclo}, we obtain
\begin{equation}
\|f_{0}\|=\|w\|_{2}, \|Q\|<\|f\|_{2}<\|Q\|_{2}+\alpha
\end{equation}

Now, by Fatou's Lemma, $\|f_{0}\|_{\dot{H}^{1}}\leq \liminf_{n}\|f_{n}\|_{\dot{H}^{1}}$.

By interpolation\footnote{Note we do not need the radial Sobolev embedding here} between $L^{2}$ and $H^{1}$, $\|f_{n}\|_{L^{4/d+2}}\rightarrow \|f_{0}\|_{L^{4/d+2}}$.

Thus, we have the condition
\begin{equation}\label{energycondition}
E(f_{0})\leq \liminf E(f_{n})\leq 0
\end{equation}
Now, to summarize, with \eqref{uniformclo}, strong convergence of $f_{n}$ to $f_{0}$ in $L^{2}$,  we have
\begin{itemize}
\item $w$ is an almost periodic solution,$$ w(t)=N^{d/2}(t)L_{t}(N(t),x) , t\geq 0, N(t)\leq 1.$$
\item $ L_{t_{n}}$ converges strongly in $L^{2}$ to $f$.
\item $f$ is in $H^{1}$, and if of nonpositive energy.
\end{itemize}
These three property is enough to derive a contradiction, we prove something slightly more general for the convenience of future work.

The proof we are left with  will not depend on that the radial property, we only need that  $w$ is almost periodic in the sense that
\begin{equation}\label{weakalmost}
w(t)=\frac{1}{\lambda^{d}(t)}L_{t}(\frac{x-x(t)}{\lambda(t)})e^{ix\xi(t)}, t\geq 0.
\end{equation}

Let $F$ be the solution to \eqref{nls} with initial data $f$, then by Theorem \ref{mr1}, $F$ will blow up in finite time $T^{+}>0$, and according to Corollary \ref{mr1cor},  there exists $\delta_{0}=\delta_{0}(F)>0$, such that for $\forall A>0$, there exists $T_{A}<T^{+}, x_{0}\in \RRR^{d}$ 
\begin{equation}\label{decouplemass}
 \int_{|x-x_{0}|\leq l_{0}} |F(T_{A},x)|^{2}\geq \delta_{0}, \int_{|x-x_{0}|\geq Al_{0}} |F(T_{A},x)|^{2}\geq \delta_{0}. 
 \end{equation}
 And note by standard local theory of \eqref{nls},
 \begin{equation}\label{finitesri}
\| F(t,x)\|_{L_{t,x}^{2(d+2)/d}([0,T_{A}]\times \RRR^{d})}\leq C_{A}<\infty.
 \end{equation}
 \begin{rem}
Note  here $\delta_{0}$ is fixed once $w$ is fixed, and $A$ can be chosen arbitrary large (by choosing $T=T_{A}$ close to $T^{+}$ enough.)
 \end{rem}
 On the other hand, since $w$ is an almost periodic solution of form \eqref{weakalmost}, we claim
\begin{lem}\label{aprigid}
 Assume $w$ is of form \eqref{weakalmost}, then, for any $\delta>0$, there exists $A=A_{\delta}$ such that if  for some $z_{0}\in \RRR^{d}, t_{0}\in \infty), h_{0}>0$ such that
 \begin{equation}\label{nondecouple}
 \int_{|x-z_{0}|\leq h_{0}}|w(t_{0})|^{2}\geq \frac{\delta}{2},
 \end{equation}
 then we must have
 \begin{equation}
 \int_{|x-z_{0}|\geq Ah_{0}}|w(t_{0},x)|^{2}dx\leq \frac{\delta}{2}.
 \end{equation}
 \end{lem}
Lemma \eqref{aprigid} will be proven in Subsection  \ref{subaprigid}, let us assume it at the moment and finish the proof of Lemma \ref{quasuper}.
 Let us fix $\delta_{0}=\delta_{0}(w)$, and picking $\delta_{1}=\delta_{0}/2$, and let $A=A_{\delta_{1}}$ as in Lemma \ref{aprigid} such that \eqref{nondecouple} holds , and as mentioned before, we can use Corollary \ref{mr1cor} to find $T=T_{A}$ such that \eqref{decouplemass} holds and we emphasize again \eqref{finitesri} holds. 
 
 Now since $L_{t_{n}} \rightarrow F(0)$ in $L^{2}$, then  for  $\forall \epsilon>0$, there exists $n_{0}=n_{0}(\epsilon)$ such that
 \begin{equation}
 \|L_{t_{n_{0}}}-F(0)\|_{2}<\epsilon
 \end{equation}
Using  the stability argument, Proposition \ref{stability},  by choosing $\epsilon$ small enough, (according to $C_{A}$), then \eqref{nls} with initial data $L_{t_{n_{0}}}$ has a solution, we call it $v$, which exists in $[0,T_{A}]$,  such that
 \begin{equation}
 \|v(T_{A})-w(T_{A})\|_{2}\leq \delta_{0}/10.
 \end{equation}
 By \eqref{decouplemass} and triangle inequality, we have
 \begin{equation}\label{minidecouple}
 \int_{|x-x_{0}|\leq l_{0}} |v(T_{A},x)|^{2}\geq \delta_{0}/2,\quad  \int_{|x-x_{0}|\geq Al_{0}} |w(T_{A},x)|^{2}\geq \delta_{0}/2. 
 \end{equation}
 
 On the other hand, since $v(t)$ solves \eqref{nls} in $[0,T_{A}]$ with initial data $L_{t_{n_{0}}}$, and $w(t)$ solves \eqref{nls} with $w({t_{n_{0}}})=\frac{1}{\lambda^{d/2}(t_{n_{0}})}L_{t_{n_{0}}}(\frac{x-\xtn}{\ltn})e^{ix\xitn}$, by the local well posedness theory (uniqueness of the solution), we have  $w$ is defined in \\
  $[t_{n_{0}},t_{n_{0}}+\ltn^{2}T_{A}]$,  and 
 \begin{equation}
 \begin{aligned}
 &w(t_{n_{0}}+\ltn^{2}T_{A},x)\\
 =&\frac{1}{\ltn^{d/2}}v\left(T_{A}, \frac{x-\xtn-2\xitn \ltn^{2}T_{A}}{\ltn}\right)e^{-i\ltn^{2}|\xitn|^{2}T_{A}}e^{ix\xitn}\\
 =&\frac{1}{\ltn^{2}}v\left(T_{A},\frac{x-\tilde{x}_{0}}{\ltn}\right)e^{i\gamma_{0}}e^{i\xitn x}.
 \end{aligned}
 \end{equation} 
 where $\tilde{x}_{0}=\xtn+2\xitn \ltn^{2}T_{A}{\ltn}$.

Let $\tilde{t}_{n_{0}}:=t_{n}+\ltn^{2}T_{A},  z_{0}=\tilde{x_{0}}+x_{0}$, and plug in \eqref{minidecouple}, we obtain
\begin{equation}
\int_{|x-z_{0}|\leq \ltn l_{0}}|w(\tilde{t}_{n_{0}},x)|^{2}\geq \delta_{0}/2, \quad \int_{|x-z_{0}|\geq A\ltn l_{0}}w(\tilde{t}_{n_{0}},x)|^{2}\geq \delta_{0}/2
\end{equation}
This contradicts Lemma \ref{aprigid}.
To finish the proof of Lemma \ref{nofastcascade}, we are left with the proof of Lemma \ref{aprigid}, which will be done in the following subsection.
\subsection{Proof of Lemma \ref{aprigid}}\label{subaprigid}
Indeed, we only need to prove Lemma \ref{aprigid} for $w(t)=L_{t}, t\geq 0.$ Since the statement of Lemma already takes space translation into account, and the $l_{0}$ takes care of the scaling, and the phase $e^{ix\xi(t)}$ plays no role in this argument.

Thus, we reduce the proof to following Lemma 
\begin{lem}
For a precomact $L^{2}$ family $\{L_{t}\}_{t\geq 0}$, (not necessarily radial),  $\forall \delta>0$ there exists $A>0$ such that
if for some $l_{0}>0$,
\begin{equation}\label{concen}
\int_{|x|\leq l_{0}} |L_{t_{0}}|^{2}\geq  \delta,
\end{equation}
then 
\begin{equation}\label{nonconcen}
\int_{|x|\geq Al_{0}}|L_{t_{0}}|^{2}\leq \delta.
\end{equation}
\end{lem}
\begin{proof}
Since $\{L_{t}\}$ is precompact in $L^{2}$, by standard approximation argument, one may without loss of generality may assume $L^{t}$ is uniformly bounded and their support are uniformly compact. Thus\eqref{concen} implies $1\lesssim l_{t_{0}}$, since $L_{t}$ is uniformly bounded. Thus,  when $A$ is large enough, clearly \eqref{nonconcen} holds.
\end{proof}

\section{Proof of Theorem \ref{main1}}\label{step5}
Let $\Phi_{1}$ , $w$ be as in Lemma \ref{secondextraction}.
Recall we have that $\{u(t_{n})\}$ admits profile decomposition with profiles $\{\phi_{j}, \{x_{j,n},\lambda_{j,n}, \xi_{j,n}, t_{j,n}\}_{n}\}_{j}$ . Since now we only consider radial solution, we indeed have $x_{j,n}=\xi_{j,n}=0$.

And as explained in Subsection \ref{finalapproximation}  we have
\begin{equation}
\|w\|_{2}=\|Q\|_{2}, \int N(t)^{3}=\infty
\end{equation}
Thus, by Lemma \ref{quathresh}, there exists a time sequence $s_{n}$ and parameter $\{\lambda_{n},\gamma_{n}\}_{n}$ such that
\begin{equation}\label{finalstrong}
\lim_{n\rightarrow \infty}\|\lambda_{n}^{d/2}w(s_{n}, \lambda_{n}x)-Qe^{-i\gamma_{n}}\|_{2}=0.
\end{equation}
We point out there is a slight abuse of notation in Lemma \ref{firstextraction}, Lemma \ref{secondextraction}, and  Lemma \ref{quathresh}. The time sequences $\{t_{n}\}_{n}$ in these lemmas are not necessarily same. From now on, we change the time sequence in Lemma \ref{quathresh} to $\{s_{n}\}_{n}$, and change the time sequence in Lemma \ref{secondextraction} to $\{l_{n}\}$.

Note by a extracting subsequence, we may without loss of generality assume the $\gamma_{n}$ in \eqref{finalstrong} satisfy $\gamma_{n}\equiv \gamma_{0}.$ Note phase symmetry is a compact symmetry in $L^{2}$.
 
We claim  \eqref{finalstrong} implies
\begin{lem}\label{firstapproximation}
There exists a time sequence $\{h_{n}\}$ such that  $\Phi_1(h_{n})$ admits a profile decomposition with profiles $\{v_{j}\}_{j}$. (We will not track the associated parameters here.) Moreover, there is a  profile, we call it $v_{1}$ such that $v_{1}=Qe^{-i\gamma_{0}}$
\end{lem}
\begin{rem}
Since $\|\Phi_{1}\|_{2}=\|w\|_{2}=\|Q\|_{2}$, this indeed implies $v_{1}$ is the only profile and one can conclude similar strong convergence results as in \eqref{finalstrong}.
\end{rem}
\begin{proof}[Proof of Lemma \ref{firstapproximation}]
Recall $P_{l_{n}}, P_{0}, w, \Phi$ in Lemma \ref{secondextraction}. Recall we have already changed the notation $t_{n}$ to $l_{n}$ in Lemma \ref{secondextraction}.
Without loss of generality, we assume in Lemma \ref{secondextraction}, the initial data of $w$ is $P_{0}$.
Indeed, since $P_{l_n}\rightarrow P_{0}$, this implies $\Phi(l_{n})$ admits a profile decomposition with only one proifle $P_{0}$, and with stability argument, Proposition \ref{stability} or Proposition \ref{nonlinearprofile},  it further implies $\{\Phi(l_{n}+\lambda_{1}(l_{n}^{2})s_{p})\}$ admits profile decompostition $w(s_p)$ for each $p$, also note $\{w(s_p)\}$ admits a profile decomposition with only one profile $Q^{e^{i\gamma_{0}}}$ by \eqref{finalstrong}, thus by Lemma \ref{double}, there exists a sequence $\{n_{p}\}_{p}$, such that $\Phi_{1}(l_{n_{p}}+\lambda(l_{n_{p}})^{2}s_{p})$ admits a profile decomposition, one of the profile is $Q{e^{i\gamma}}$. 
\end{proof}

To finish the proof of Theorem \ref{main1}, let $h_{p}$ be as in Lemma \ref{firstapproximation}, one simply argues as in the proof of Lemma \ref{firstapproximation} we just did, and use Lemma \ref{double} to  do double profile decompositon for $\{u(t_{n}+\lambda_{1,n}^{2}h_{p})\}_{n,p}$ .

\section{Acknowledgment}
The author is very grateful to his adviser, Gigliola Staffilani for consistent support and encouragement and very careful reading of the material, and a lot of helpful discussion. Part  of  the material is based upon work supported by the National Science Foundation under Grant No. 0932078000 while the author was in residence at the Mathematical Sciences Research Institute in Berkeley, California, during Fall 2015 semester. The author benefits a lot in his stay in MSRI. Part of the material is finished when the author was attending the Nonlinear Evolution Problem workshop in MFO, Oberwolfach Research Institute for Mathematics, which is a good place to do research.
\appendix
\section{Proof of Theorem \ref{aux}}\label{sectionaux}
As shown in the proof of Lemma \ref{firstextraction}, such solution $u$ is of form
\begin{equation}
u(t,x)=\frac{1}{\lambda^{d/2}(t)}V_{t}(\frac{x-x(t)}{\lambda(t)})e^{ix\xi(t)}.
\end{equation}
$\{V_{t}\}$ is a compact $L^{2}$ family and$ \|V_{t}\|_{2}=\|Q\|_{2}$.  We assume $u$ blows up in finite time $T$. 
To prove Theorem \ref{aux}, we need only exclude the possibility
\begin{equation}\label{nono}
\lambda(t)\gtrsim  (T-t)^{2/3-}.
\end{equation}

Assume \eqref{nono} holds, do a (inverse) pesudo conformal transformation of $u$, we will get a solution $v$ such that
\begin{equation}
v(\tau,y)=\frac{1}{\tau^{d/2}}u(T-\frac{1}{\tau},\frac{y}{\tau})e^{i\frac{1}{4}\frac{|y|^{2}}{\tau}}= N^{d/2}(\tau)W_{\tau}(N(\tau)y+y(\tau))e^{i\tilde{\xi}(\tau)y}
\end{equation}
Where 
\begin{equation}
\begin{aligned}
&N(\tau)=\frac{1}{\tau}\frac{1}{\lambda(T-\frac{1}{\tau})},\\
&W_{\tau}(z)=V_{T-\frac{1}{\tau}}(z)e^{i\frac{1}{4}\frac{1}{N^{2}(\tau)\tau}|z|^{2}}e^{-i\left(x(T-\frac{1}{\tau})\right)^{2}\tau},\\
&y(\tau)=\tau x(T-\frac{1}{\tau}), \tilde{\xi}(\tau)=\frac{\xi(T-\frac{1}{\tau})}{\tau}+2x(T-\frac{1}{\tau})
\end{aligned}
\end{equation}
Note the exact value of $y(\tau), \tilde{\xi}(\tau)$ actually do not matter.

We obtain by \eqref{nono} 
\begin{equation}
N(\tau)\leq \frac{1}{\tau}\lesssim \frac{1}{\tau}\tau^{2/3-}\lesssim (\frac{1}{\tau})^{1/3+}
\end{equation}
On the other hand, one has the scaling lower bound
\begin{equation}
\lambda(t)\lesssim (T-t)^{1/2}
\end{equation}
Thus,$ N(\tau)\geq \tau^{-1/2}$,

This further implies $W_{\tau}$ is also precompact in $L^{2}$, since multiplying  $e^{i\frac{1}{4}\frac{1}{N^{2}(\tau)\tau}|z|^{2}}$  is a compact perturbation.  This contradicts Lemma \ref{nofastcascade}, since now we have $\int_{0}^{\infty} N^{3}(\tau)<\infty$.

\bibliographystyle{alpha}
\bibliography{BG}
\end{document}